\documentclass[11pt]{amsart}

\usepackage[english]{babel}
\usepackage{bm}
\usepackage{fullpage}
\usepackage{amssymb,amsmath,latexsym}
\usepackage{color}
\usepackage{comment}

\newtheorem{lemma}{Lemma}

\newtheorem{theorem}{Theorem}
\newtheorem{corollary}{Corollary}

\theoremstyle{remark}
\newtheorem{conjecture}{Conjecture}

\newtheorem{remark}{Remark}

\newcommand{\cal}{\mathcal}
\def\ds{\displaystyle}

\def\Z{\mathbb{Z}}
\def\x{\boldsymbol{x}}
\def\y{\boldsymbol{y}}

\def\G{\cal{G}}

\def\hyp{\cal{H}}

\def\alt{\operatorname{alt}}
\def\salt{\operatorname{salt}}
\def\supp{\operatorname{supp}}

\def\KG{\operatorname{KG}}
\def\cd{\operatorname{cd}}
\def\id{\operatorname{id}}

\title{Hedetniemi's conjecture for Kneser hypergraphs}

\author{Hossein Hajiabolhassan} 
\address{H. Hajiabolhassan, Department of Applied Mathematics and Computer Science,Technical University of Denmark, DK-{\rm 2800} Lyngby, Denmark, and
Department of Mathematical Sciences, Shahid Beheshti University, G.C., P.O. Box {\rm 19839-63113}, Tehran, Iran}
\email{hhaji@sbu.ac.ir}

\author{Fr\'ed\'eric Meunier}
\address{F. Meunier, Universit\'e Paris Est, CERMICS, 77455 Marne-la-Vall\'ee CEDEX, France}
\email{frederic.meunier@enpc.fr}


\begin{document}

\maketitle

\begin{abstract}
One of the most famous conjecture in graph theory is Hedetniemi's conjecture stating that the chromatic number of the categorical product of graphs is the minimum of their chromatic numbers. Using a suitable extension of the definition of the categorical product, Zhu proposed in 1992 a similar conjecture for hypergraphs. We prove that Zhu's conjecture is true for the usual Kneser hypergraphs of same rank. It provides to the best of our knowledge the first non-trivial and explicit family of hypergraphs with rank larger than two satisfying this conjecture (the rank two case being Hedetniemi's conjecture). We actually prove a more general result providing a lower bound on the chromatic number of the categorical product of any Kneser hypergraphs as soon as they all have same rank. We derive from it new families of graphs satisfying Hedetniemi's conjecture. The proof of the lower bound relies on the $Z_p$-Tucker lemma.
\end{abstract}

\section{Introduction}

\subsection{Categorical product and coloring} 
Let $G=(V,E)$ and $G'=(V',E')$ be two graphs. Their {\em categorical product}, denoted $G\times G'$, is the graph defined by $$\begin{array}{rcl}V(G\times G') & = & V\times V' \\ E(G\times G') & = &\{\{(u,u'),(v,v')\}:\,\{u,v\}\in E,\,\{u',v'\}\in E'\}.\end{array}$$
Hedetniemi's conjecture -- one of the most intriguing conjecture in graph theory -- states that the chromatic number of the categorical product of two graphs is the minimum of their chromatic numbers. Hedetniemi's conjecture has been verified in many cases, but the general case is still open. Tardif~\cite{Ta08} and Zhu~\cite{Zh98} provide extensive surveys of this topic. There exists also a categorical product for hypergraphs, defined by D\"ofler and Waller~\cite{DoWa80} in 1980. Using this definition of the categorical product of hypergraphs, Zhu~\cite{Zh92} conjectured in 1992 that a generalization of Hedetniemi's conjecture holds for hypergraphs as well.

Let $\G=(V,E)$ and $\G'=(V',E')$ be two hypergraphs. Their {\em categorical product}, denoted $\G\times\G'$, is the hypergraph defined by 
$$\begin{array}{rcl}
V(\G\times\G') & = & V\times V' \\
E(\G\times\G') & = &\left\{\{(u_1,u_1'),\ldots,(u_r,u_r')\}:\,r\in\Z_+,\,\{u_1,\ldots,u_r\}\in E,\,\{u_1',\ldots,u_r'\}\in E'\right\},
\end{array}$$
where the $u_i$'s and the $u_i'$'s do not need to be distinct. In other words, a subset of $V\times V'$ is an edge of $\G\times\G'$ if its projection on the first component is an edge of $\G$ and its projection on the second component is an edge of $\G'$. Note that this product can be made associative by a natural identification and thus defined for more than two hypergraphs.

We recall that a {\em proper coloring} of a hypergraph is an assignment of colors to the vertices so that there are no monochromatic edges. The {\em chromatic number} of a hypergraph $\G$, denoted $\chi(\G)$, is the minimum number of colors a proper coloring of $\G$ may have. If the chromatic number is $k$ or less, the hypergraph is {\em $k$-colorable}.
\begin{conjecture}[\cite{Zh92}]\label{conj:zhu}
Let $\G$ and $\G'$ be two hypergraphs. We have 
$$\chi(\G\times\G')=\min\{\chi(\G),\chi(\G')\}.$$
\end{conjecture}
The chromatic number of the product is at most the chromatic number of each of the hypergraphs, since a coloring of any of them provides a coloring of the product. The difficult part in this conjecture is thus the reverse inequality.

When $\G$ and $\G'$ are graphs, their categorical product according to the definition for hypergraphs is not a graph anymore, since it will in general contain edges of cardinality four. Anyway, the chromatic number does not depend on the definition we take for the categorical product, as it can be easily checked (further explanations are given in the next paragraph). Thus Conjecture~\ref{conj:zhu} is a true generalization of Hedetniemi's conjecture. 

We take this remark as an opportunity to give another point of view on this categorical product for hypergraphs. Let $e\in E$ and $e'\in E'$. Consider the set of all simple bipartite graphs with no isolated vertices and with color classes $e$ and $e'$. The edges of the product $\G\times\G'$ obtained from $e$ and $e'$ are in one-to-one relation with these bipartite graphs. When $\G$ and $\G'$ are graphs, these bipartite graphs are those with two vertices in each color class and no isolated vertices. Such bipartite graphs all have a perfect matching: in other words, any edge of the product of the graphs seen as hypergraphs contains an edge of the product with the usual definition of the categorical product of graphs. It explains why the chromatic number of the categorical product of two graphs does not depend on the definition we take for the categorical product. 

\subsection{Kneser hypergraphs}
In 1976, Erd\H{o}s~\cite{Er76} initiated the study of {\em Kneser hypergraphs} $\KG^r(\hyp)$ defined for a hypergraph $\hyp=(V(\hyp),E(\hyp))$ and an integer $r\geq 2$ by
$$\begin{array}{rcl}
V(\KG^r(\hyp)) & = & E(\hyp) \\
E(\KG^r(\hyp)) & = & \{\{e_1,\ldots,e_r\}:e_1,\ldots,e_r\in E(\hyp),\,e_i\cap e_j=\emptyset\mbox{ for all $i,j$ with $i\neq j$}\}.
\end{array}$$ 
These hypergraphs enjoy several properties, which are interesting from both graph theoretical and set theoretical point of views,  especially when $r=2$, i.e. when we are dealing with {\em Kneser graphs}. Among many references dealing with the Kneser hypergraphs, one can cite~\cite{AlDrLu09,LaZi07,Zi02}. For Kneser graphs, there are much more references, see~\cite{Ha10,Lo79,SiTa06,SiTa07,St76,VaVe05} among many of them. 

There are also ``usual'' Kneser hypergraphs, which are obtained with $\hyp=([n],{{[n]}\choose k})$. They are denoted $\KG^r(n,k)$. In the present paper, we prove that Conjecture~\ref{conj:zhu} is true for the usual Kneser hypergraphs, which provides to the best of our knowledge the first non-trivial and explicit family of hypergraphs not being graphs for which Conjecture~\ref{conj:zhu} is true. We actually prove a more general result involving the {\em colorability defect} of a hypergraph.  The $r$-colorability defect of a hypergraph $\hyp$, introduced by Dol'nikov~\cite{Do88} for $r=2$ and by K\v{r}\'i\v{z}~\cite{Kr92,KR00} for any $r$, is denoted $\cd_r(\hyp)$ and is the minimum number of vertices to be removed from $\hyp$ so that the remaining induced subhypergraph is $r$-colorable. The {\em subhypergraph of $\hyp$ induced by a set $X$}  is denoted $\hyp[X]$ and is the hypergraph with vertex set $X$ and with edge set $\{e\in E(\hyp):\,e\subseteq X\}$ (note that this definition of an induced subhypergraph departs from the usual way to define it). We have thus
  $$\cd_r(\hyp)=\min\{|Y|:\,Y\subseteq V(\hyp)\mbox{ and } \chi(\hyp[V(\hyp)\setminus Y])\leq r\}.$$ 
  K\v{r}\'i\v{z}~\cite{Kr92,KR00} proved that $\chi(\KG^r(\hyp))\geq\cd_r(\hyp)$. The $r$-colorability defect has then been used by other authors as a tool for exploring further properties of coloring of graphs and hypergraphs~\cite{Me14,SiTa07,Zi02}.

\subsection{Main results}
In this paper, we prove the following theorem.

\begin{theorem}\label{thm:main}
Let $\hyp_1,\ldots,\hyp_t$ be hypergraphs and $r\geq 2$ be a positive integer. If none of the $\hyp_{\ell}$'s have $\emptyset$ as an edge, then 
$$\chi\left(\KG^r(\hyp_1)\times\cdots\times\KG^r(\hyp_t)\right)\geq\frac{1}{r-1}\min_{\ell=1,\ldots,t}\cd_r(\hyp_{\ell}).$$
\end{theorem}

The case $t=1$ is K\v{r}\'i\v{z}'s theorem.

Note that the product in this theorem involves an arbitrary number of Kneser hypergraphs, while in Conjecture~\ref{conj:zhu} only two hypergraphs are involved. The reason is that the categorical product of hypergraphs is associative and if the conjecture holds for two hypergraphs, it can be straightforwardly extended to the product of any number of hypergraphs. On the contrary, it is easy to see that Theorem~\ref{thm:main} for only two Kneser hypergraphs does not imply its correctness for any number of Kneser hypergraphs.

We will actually see that a stronger result involving the {\em $r$-alternation number} instead of the $r$-colorability defect holds. The {\em $r$-alternation number} of a hypergraph has been introduced by Alishahi and Hajiabolhassan~\cite{AlHa13} and it provides better lower bounds for Kneser hypergraphs. The definition is given later in the paper.

The fact that Conjecture~\ref{conj:zhu} is true for the usual Kneser hypergraphs is obtained via the easy equality $$\cd_r\left([n],{{[n]}\choose k}\right)=n-r(k-1)$$ and a theorem by Alon, Frankl, and Lov\'asz~\cite{AlFrLo86} stating that $$\chi(\KG^r(n,k))=\left\lceil\frac{n-r(k-1)}{r-1}\right\rceil.$$

\begin{corollary}\label{cor:main}
Let $r\geq 2$ be an integer and let $n_1\ldots,n_t$ and $k_1,\ldots,k_t$ be positive integers such that $n_{\ell}\geq rk_{\ell}$ for $\ell=1,\ldots,t$. We have $$\chi\left(\KG^r(n_1,k_1)\times\cdots\times\KG^r(n_t,k_t)\right)=\min_{\ell=1,\ldots,t}\chi\left(\KG^r(n_{\ell},k_{\ell})\right).$$
\end{corollary}

If Conjecture~\ref{conj:zhu} is true, then the equality of Corollary~\ref{cor:main} is also true when the Kneser hypergraphs have not the same rank $r$. However, we were not able to find a proof of this more general result. When $r=2$, Corollary~\ref{cor:main} implies the already known fact that Hedetniemi's conjecture is true for the usual Kneser graphs~\cite{Do09,He79,VaVe06}.

\section{Proof of the main result}\label{sec:main}

\subsection{Main steps of the proof}

The proof consists in proving the following two lemmas. Their combination provides a proof of Theorem~\ref{thm:main}. They are respectively proved in Section~\ref{subsec:nonprime} and in Section~\ref{subsec:comb_topo}.

\begin{lemma}\label{lem:nonprime}
Let $r'$ and $r''$ be two positive integers. If Theorem~$\ref{thm:main}$ holds for both $r'$ and $r''$, then it holds also for $r=r'r''$.
\end{lemma}

\begin{lemma}\label{lem:main}
Let $\hyp_1,\ldots,\hyp_t$ be hypergraphs and $p$ be a prime number. If none of the $\hyp_{\ell}$'s have $\emptyset$ as an edge, then 
$$\chi\left(\KG^p(\hyp_1)\times\cdots\times\KG^p(\hyp_t)\right)\geq\frac{1}{p-1}\min_{\ell=1,\ldots,t}\cd_p(\hyp_{\ell}).$$
\end{lemma}

\begin{proof}[Proof of Theorem~$\ref{thm:main}$]
The theorem is a direct consequence of Lemmas~\ref{lem:nonprime} and~\ref{lem:main}.
\end{proof}

\subsection{Reduction to the case when $r$ is a prime number}\label{subsec:nonprime}

The proof of Lemma~\ref{lem:nonprime} is based on the following lemma. Given a hypergraph $\hyp$ and two positive integers $s$ and $C$, we define a new hypergraph $\mathcal{T}_{\hyp,C,s}$ by 
$$\begin{array}{rcl}
V(\mathcal{T}_{\hyp,C,s}) & = & V(\hyp) \\
E(\mathcal{T}_{\hyp,C,s}) & = & \{A\subseteq V(\hyp):\,\cd_s(\hyp[A])>(s-1)C\}.
\end{array}$$ 
\begin{lemma}\label{lem:th}
The following inequality holds for any positive integer $r$ $$\cd_{rs}(\hyp)\leq r(s-1)C+\cd_r(\mathcal{T}_{\hyp,C,s}).$$
\end{lemma}

\begin{proof}
Let $A\subseteq V(\hyp)$ be such that $|A|=|V(\hyp)|-\cd_r(\mathcal{T}_{\hyp,C,s})$ and such that $\mathcal{T}_{\hyp,C,s}[A]$ is $r$-colorable. The existence of such an $A$ is ensured by the definition of $\cd_r(\mathcal{T}_{\hyp,C,s})$. Consider any proper coloring of $\mathcal{T}_{\hyp,C,s}[A]$ with $r$ colors and denote by $X_j\subseteq A$ the set of vertices of color $j$. Since the coloring is proper, $X_j$ is never an edge of $\mathcal{T}_{\hyp,C,s}$. It implies that we have $\cd_s(\hyp[X_j])\leq (s-1)C$ for all $j\in[r]$. We can thus remove at most $(s-1)C$ vertices from each $X_j$ and get $Y_j\subseteq X_j$ with $\chi(\hyp[Y_j])\leq s$. There is thus a proper coloring of $\hyp[\bigcup_{j=1}^rY_j]$ with at most $rs$ colors and $\hyp[\bigcup_{j=1}^rY_j]$ is a subhypergraph of $\hyp$ obtained by the removal of at most $r(s-1)C+\cd_r(\mathcal{T}_{\hyp,C,s})$ vertices.
\end{proof}

\begin{proof}[Proof of Lemma~$\ref{lem:nonprime}$]
Assume for a contradiction that there is proper coloring $c$ of $\KG^{r'r''}(\hyp_1)\times\cdots\times\KG^{r'r''}(\hyp_t)$ with $C$ colors such that $\cd_{r'r''}(\hyp_{\ell})>(r'r''-1)C$ for all $\ell$. Lemma~\ref{lem:th} ensures then that each $\mathcal{T}_{\hyp_{\ell},C,r'}$ has at least one edge. Let $(e_1,\ldots,e_t)\in\mathcal{T}_{\hyp_1,C,r'}\times\cdots\times\mathcal{T}_{\hyp_t,C,r'}$. The inequality $\chi(\KG^{r'}(\hyp_1[e_1])\times\cdots\times\KG^{r'}(\hyp_t[e_t]))>C$ holds because Theorem~\ref{thm:main} holds for $r'$. There is thus a monochromatic edge in $\KG^{r'}(\hyp_1[e_1])\times\cdots\times\KG^{r'}(\hyp_t[e_t])$.
We define $h(e_1,\ldots,e_t)$ to be the color of such an edge. Since $c$ is a proper coloring of $\KG^{r'r''}(\hyp_1)\times\cdots\times\KG^{r'r''}(\hyp_t)$, the map $h$ is a proper coloring of $\KG^{r''}(\mathcal{T}_{\hyp_1,C,r'})\times\cdots\times\KG^{r''}(\mathcal{T}_{\hyp_t,C,r'})$ with $C$ colors. Since Theorem~\ref{thm:main} holds for $r''$, we have
$$C\geq \frac{1}{r''-1}\min_{\ell=1,\ldots,t}\cd_{r''}(\mathcal{T}_{\hyp_{\ell},C,r'}).$$ Lemma~\ref{lem:th} implies via a direct calculation that $(r'r''-1)C\geq\min_{\ell\in[t]}\cd_{r'r''}(\hyp_{\ell})$, which is in contradiction with the starting assumption.
\end{proof}

\subsection{Proof of the main result when $r$ is a prime number}\label{subsec:comb_topo}

The proof of Lemma~\ref{lem:main} makes use of a ``$Z_p$-Tucker lemma'' proposed in~\cite{Me11} as a slight generalization of a ``$Z_p$-Tucker lemma'' used by Ziegler~\cite{Zi02} in a combinatorial proof of K\v{r}\'i\v{z}'s theorem (proof inspired by Matou\v{s}ek's proof of the special case of usual Kneser graphs~\cite{Mat04}). Before stating it, we introduce some notations.

We denote by $Z_r$ the cyclic and multiplicative group of the $r$th roots of unity. We denote by $\omega$ one of its generators. Let $\x=(x_1,\ldots,x_n)$ and $\x'=(x'_1,\ldots,x'_n)$ be two vectors of $(Z_r\cup\{0\})^{n}$. By the notation ``$\x\subseteq\x'$'', we mean that the following implication holds for all $i\in[n]$ $$x_i\neq 0\quad\Longrightarrow\quad x'_i=x_i.$$ The number of nonzero components of a vector $\x$ is denoted $|\x|$. We introduce also the notation $\supp_j(\x)$ for the set of indices $i$ such that $x_i=\omega^j$. Note that we have $|\x|=\sum_{j=1}^r|\supp_j(\x)|$.

For $\x\in(Z_r\cup\{0\})^{n}\setminus\{(0,\ldots,0)\}$, define $\omega\cdot\x=(\omega x_1,\ldots,\omega x_n)$. It defines a free action of $Z_r$ on $(Z_r\cup\{0\})^{n}\setminus\{(0,\ldots,0)\}$. For $(\omega^j,k)\in Z_r\times[m]$, where $m$ is a positive integer, we define $\omega\cdot(\omega^j,k)$ to be $(\omega^{j+1},k)$. It defines a free action of $Z_r$ on $Z_r\times[m]$. A map between two sets on which $Z_r$ acts freely is {\em equivariant} if it commutes with the action of $Z_r$. 

\begin{lemma}[$Z_p$-Tucker lemma]\label{lem:zptucker} Let $p$ be a prime number, $n,m\geq 1$, $\alpha\leq m$ and let
$$\begin{array}{lccc}\lambda:
& (Z_p\cup\{0\})^{n}\setminus\{(0,\ldots,0)\} &
\longrightarrow & Z_p\times [m] \\
& \x & \longmapsto & (s(\x),v(\x))
\end{array}$$ be a $Z_p$-equivariant map satisfying the following two properties:
\begin{itemize}
\item for all $\x^{(1)}\subseteq \x^{(2)}\in(Z_p\cup\{0\})^{n}\setminus\{(0,\ldots,0)\}$, if $v(\x^{(1)})=v(\x^{(2)})\leq\alpha$, then $s(\x^{(1)})=s(\x^{(2)})$;
\item
for all
$\x^{(1)}\subseteq \x^{(2)}\subseteq \cdots \subseteq \x^{(p)}\in(Z_p\cup\{0\})^{n}\setminus\{(0,\ldots,0)\}$, if $v(\x^{(1)})=v(\x^{(2)})=\cdots=v(\x^{(p)})\geq\alpha+1$, then the $s(\x^{(i)})$ are not pairwise distinct for $i=1,\ldots,p$. 
\end{itemize}

Then $\alpha+(m-\alpha)(p-1)\geq n$.
\end{lemma}

The second ingredient of the proof is a $Z_p$-equivariant map $$\varepsilon:(2^{Z_p}\times\cdots\times 2^{Z_p})\setminus\left(\{\emptyset,Z_p\}\times\cdots\times\{\emptyset,Z_p\}\right)\longrightarrow Z_p,$$ where we assume $p$ to be a prime number. 
For $B\subseteq Z_p$, we define $\omega\cdot B$ to be $\cup_{b\in B}\{\omega b\}$. Since $p$ is a prime number, it defines a free action of $Z_p$ on $2^{Z_p}\setminus\{\emptyset,Z_p\}$. We extend the action of $Z_p$ on $2^{Z_p}\times\cdots\times 2^{Z_p}\setminus(\{\emptyset,Z_p\}\times\cdots\times\{\emptyset,Z_p\})$ by defining $\omega\cdot(B_1,\ldots,B_t)$ to be $(\omega\cdot B_1,\ldots,\omega\cdot B_t)$. This action is again free. We can thus define such a $Z_p$-equivariant map $\varepsilon(\cdot)$, which gives a
``sign'' to each $t$-tuple $(B_1,\ldots,B_t)\in 2^{Z_p}\times\cdots\times 2^{Z_p}$ such that at least one of the $B_{\ell}$'s is not in $\{\emptyset,Z_p\}$.

\begin{proof}[Proof of Lemma~$\ref{lem:main}$]
Without loss of generality, we assume that
$$\cd_p(\hyp_1)=\min_{\ell=1,\ldots,t}\cd_p(\hyp_{\ell}).$$ 

Denote by $V_{\ell}$ and by $E_{\ell}$ respectively the vertex set and the edge set of $\hyp_{\ell}$. The cardinality of $V_{\ell}$ is denoted by $n_{\ell}$ and we arbitrarily identify $V_{\ell}$ and $[n_{\ell}]$. Let $c:E_1\times\cdots\times E_t\rightarrow[C]$ be a proper coloring of $\KG^p(\hyp_1)\times\cdots\times\KG^p(\hyp_t)$ with $C$ colors.
We endow $E_1\times\cdots\times E_t$ with an arbitrary total order $\preceq$ such that $(S_1,\ldots,S_t)\preceq (T_1,\ldots,T_t)$ if $c(S_1,\ldots,S_t)<c(T_1,\ldots,T_t)$. We shall apply the $Z_p$-Tucker lemma (Lemma~\ref{lem:zptucker}) with $n=\sum_{\ell=1}^tn_{\ell}$, $\alpha=n-\cd_p(\hyp_1)+p-1$, and $m=\alpha+C-1$. The lower bound on $C$, and thus on the chromatic number, will be a direct consequence of the inequality $\alpha+(m-\alpha)(p-1)\geq n$. To that purpose, we define a map $\lambda:\x\in(Z_p\cup\{0\})^{n}\mapsto(s(\x),v(\x))\in Z_p\times[m]$. 

Let $\x\in(Z_p\cup\{0\})^{n}\setminus\{(0,\ldots,0)\}$. From this $\x$, we define $t$ vectors $\y_1,\ldots,\y_t$ as follows: the vector $\y_1$ is the vector made of the first $n_1$ entries of $\x$, the vector $\y_2$ is the vector made of the following $n_2$ entries of $\x$,... and the vector $\y_t$ is made of the last $n_t$ entries of $\x$. The vector $\y_{\ell}$ is thus an element of $(Z_p\cup\{0\})^{n_{\ell}}$. We define also $A_{\ell}$ to be the set of $\omega^j$ such that $\supp_j(\y_{\ell})$ contains at least one edge of $E_{\ell}$. 
Two cases have to be distinguished. \\

\noindent{\bf First case: $A_{\ell}\neq Z_p$ for at least one $\ell\in[t]$.} We define
$$\begin{array}{rcl} v(\x) & = & \ds{\sum_{\ell:\,A_{\ell}\in\{\emptyset,Z_p\}}|\y_{\ell}|} \\ \\&&+\ds{\sum_{\ell:\,A_{\ell}\notin\{\emptyset,Z_p\}}\left(|A_{\ell}|+\max\left\{|\widetilde{\y}_{\ell}|:\,\widetilde{\y}_{\ell}\subseteq\y_{\ell}\mbox{ and }E(\hyp_{\ell}[\supp_j(\widetilde{\y}_{\ell})])=\emptyset\mbox{ for all $j$})\right\}\right)}.\end{array}$$
If $A_{\ell}\in\{\emptyset,Z_p\}$ for all $\ell\in[t]$, we define $s(\x)$ to be the first nonzero entry of $\x$. Otherwise, we define $s(\x)$ to be $\varepsilon(A_1,\ldots,A_t)$. We always have  in this case $1\leq v(\x)\leq\alpha$.\\

\noindent {\bf Second case: $A_{\ell}=Z_p$ for all $\ell\in[t]$.} For each $j$, there is at least one $t$-tuple $(S_1(j),\ldots,S_t(j))\in E_1\times\cdots\times E_t$ with $S_{\ell}(j)\subseteq \supp_j(\y_{\ell})$ for all $\ell\in[t]$. Among all such $t$-tuples and over all $j$, select the one that is minimal for $\preceq$. We define then $v(\x)$ to be $\alpha+c(S_1(j),\ldots,S_t(j))$ and $s(\x)$ to be $\omega^j$ with $j\in[p]$ such that $(S_1(j),\ldots,S_t(j))$ has been chosen. Note that because of the definition of $\preceq$ and because $c(\cdot)$ is a proper coloring, $c(S_1(j),\ldots,S_t(j))\neq C$, and thus $\alpha+1\leq v(\x)\leq\alpha+C-1$, as required. \\

The fact that such a map $\lambda$ is $Z_p$-equivariant can be easily checked. It remains to check that this map satisfies the two required properties for the application of $Z_p$-Tucker lemma.

Let $\x\in(Z_p\cup\{0\})^{n}\setminus\{(0,\ldots,0)\}$ such that $v(\x)\leq\alpha$. We are necessarily in the first case. We seek the condition under which making one of its zero components a nonzero one does not modify $v(\x)$.  Such a transformation cannot concern an entry in a $\y_{\ell}$ with $\ell$ such that $A_{\ell}=\emptyset$, nor it can concern an entry in a $\y_{\ell}$ with $\ell$ such that $A_{\ell}=Z_p$. It remains to see what happens when the transformation concerns an entry in a $\y_{\ell}$ with $\ell$ such that $A_{\ell}\notin\{\emptyset,Z_p\}$. If $A_{\ell}$ increases its cardinality by one because of this transformation, then $v(\x)$ increases by at least one. Hence, making a zero component of $\x$ a nonzero one modifies $v(\x)$ only if at least one of the $A_{\ell}$ is not in $\{\emptyset,Z_p\}$ and if none of the $A_{\ell}$'s are modified by this transformation. Therefore, making a zero component of $\x$ an nonzero one while not modifying $v(\x)$ does not modify $s(\x)$ as well. For $\x\subseteq \x'\in(Z_p\cup\{0\})^{n}\setminus\{(0,\ldots,0)\}$ such that $v(\x')\leq\alpha$, we necessarily have $v(\x)\leq v(\x')$. Thus, according to the discussion we just proposed, if $v(\x^{(1)})=v(\x^{(2)})\leq\alpha$ with $\x^{(1)}\subseteq\x^{(2)}$, we have $s(\x^{(1)})=s(\x^{(2)})$.

Let $\x^{(1)}\subseteq \x^{(2)}\subseteq \cdots \subseteq \x^{(p)}\in(Z_p\cup\{0\})^{n}\setminus\{(0,\ldots,0)\}$ be such that $v(\x^{(1)})=v(\x^{(2)})=\cdots=v(\x^{(p)})\geq\alpha+1$. We are necessarily in the second case. Define $(S_1^{(i)},\ldots,S_t^{(i)})$ to be the $t$-tuple used in the computation of $v(\x^{(i)})$. Suppose for a contradiction that the $s(\x^{(i)})$'s are pairwise distinct for $i=1,\ldots,p$. Then the $(S_1^{(i)},\ldots,S_t^{(i)})$'s form an edge of $\KG^p(\hyp_1)\times\cdots\times\KG^p(\hyp_t)$ while getting the same color by $c(\cdot)$ since $v(\x^{(1)})=v(\x^{(2)})=\cdots=v(\x^{(p)})$. It contradicts the fact that $c(\cdot)$ is a proper coloring. Hence, the $s(\x^{(i)})$ are not pairwise distinct for $i=1,\ldots,p$.
\end{proof}

\section{Improvements via alternation number}

\subsection{Main theorem involving the alternation number}

An {\it alternating sequence} is a sequence $x_1,x_2,\ldots,x_m\in Z_r$ such that any two consecutive terms are different. For any $\x=(x_1,x_2,\ldots,x_n)\in (Z_r\cup\{0\})^{n}$ and any permutation $\pi\in\mathcal{S}_n$, we denote by $\alt_{\pi}(\x)$ the maximum length of an alternating subsequence of the sequence $x_{\pi(1)},\ldots,x_{\pi(n)}$. Note that by definition this subsequence uses only elements of $Z_r$. In particular, if $\x=(0,\ldots,0)$, then $\alt_{\pi}(\x)=0$ for any permutation $\pi$.

Let $\hyp=(V,E)$ be a hypergraph with $n$ vertices. We identify $V$ and $[n]$. The {\em $r$-alternation number} $\alt_r(\hyp)$ of $\hyp$ is defined as
$$\alt_r(\hyp)=\min_{\pi\in\mathcal{S}_n}\max\{\alt_{\pi}(\x):\,\x\in(Z_r\cup\{0\})^n\mbox{ with }\,E(\hyp[\supp_j(\x)])=\emptyset\mbox{ for $j=1,\ldots,r$}\}.$$
In other words, for each permutation $\pi$ of $[n]$, we choose $\x$ such that $\alt_{\pi}(\x)$ is maximal while none of the $\supp_j(\x)$'s contain an edge of $\hyp$; then, we take the permutation for which this quantity is minimal.
This number does not depend on the way $V$ and $[n]$ have been identified.

We clearly have $|V|-\alt_r(\hyp)\geq\cd_r(\hyp)$. Theorem~\ref{thm:main} can now be improved with the help of the alternation number.

\begin{theorem}\label{thm:main_bis}
Let $\hyp_1,\ldots,\hyp_t$ be hypergraphs and $r\geq 2$ be a positive integer. If none of the $\hyp_{\ell}$'s have $\emptyset$ as an edge, then 
$$\chi\left(\KG^r(\hyp_1)\times\cdots\times\KG^r(\hyp_t)\right)\geq\frac{1}{r-1}\min_{\ell=1,\ldots,t}\left(|V(\hyp_{\ell})|-\alt_r(\hyp_{\ell})\right).$$
\end{theorem}

This theorem implies that Conjecture~\ref{conj:zhu} holds for Kneser hypergraphs $\KG^r(\hyp)$ of same rank whose chromatic number equals $\left\lceil\frac{1}{r-1}\left(|V(\hyp)|-\alt_r(\hyp)\right)\right\rceil$. The {\em $2$-stable Kneser hypergraphs} are such hypergraphs provided that $r-1$ do not divide $n-k$ (proved in~\cite{AlHa13}). These hypergraphs, denoted $\KG^r(n,k)_{2\mbox{\tiny{-stab}}}$, are the Kneser hypergraphs whose vertices are the $k$-subsets $A$ of $[n]$ no elements of which are cyclically adjacent (if $i\neq i'$ are both in $A$, then $2\leq |i-i'|\leq n-2$). Other examples of such hypergraphs $\hyp$ having their chromatic number equaling $\left\lceil\frac{1}{r-1}\left(|V(\hyp)|-\alt_r(\hyp)\right)\right\rceil$ are given by some {\em multiple Kneser hypergraphs}, see \cite{AlHa13} for more details.

The proof of Theorem~\ref{thm:main_bis} follows the same lines as the proof of Theorem~\ref{thm:main} given in Section~\ref{sec:main}. We have lemmas similar to Lemmas~\ref{lem:nonprime} and~\ref{lem:main}.

\begin{lemma}\label{lem:nonprime_bis}
Let $r'$ and $r''$ be two positive integers. If Theorem~$\ref{thm:main_bis}$ holds for both $r'$ and $r''$, then it holds also for $r=r'r''$.
\end{lemma}

\begin{lemma}\label{lem:main_bis}
Let $\hyp_1,\ldots,\hyp_t$ be hypergraphs and $p$ be a prime number. If none of the $\hyp_{\ell}$'s have $\emptyset$ as an edge, then 
$$\chi\left(\KG^p(\hyp_1)\times\cdots\times\KG^p(\hyp_t)\right)\geq\frac{1}{p-1}\min_{\ell=1,\ldots,t}\left(|V(\hyp_{\ell})|-\alt_p(\hyp_{\ell})\right).$$
\end{lemma}

Lemma~\ref{lem:nonprime_bis} can be proved via a technical lemma similar to Lemma~\ref{lem:th}. We define $\widetilde{\mathcal{T}}_{\hyp,C,s}$ by
$$\begin{array}{rcl}
V(\widetilde{\mathcal{T}}_{\hyp,C,s}) & = & V(\hyp) \\
E(\widetilde{\mathcal{T}}_{\hyp,C,s}) & = & \{A\subseteq V(\hyp):\,|A|-\alt_s(\hyp[A])>(s-1)C\}.
\end{array}$$ 

\begin{lemma}\label{lem:th_bis}
The following inequality holds for any positive integer $r$ $$\alt_r(\widetilde{\mathcal{T}}_{\hyp,C,s})\leq r(s-1)C+\alt_{rs}(\hyp).$$
\end{lemma}

We omit the proofs of Lemmas~\ref{lem:nonprime_bis} and~\ref{lem:th_bis} since they are very similar to the proofs given in Section~\ref{subsec:nonprime}. The proof of Lemma~\ref{lem:main_bis} is also almost identical to the proof of Lemma~\ref{lem:main}: it uses the $Z_p$-Tucker lemma (Lemma~\ref{lem:zptucker}) and the sign $\varepsilon(\cdot)$ of Section~\ref{subsec:comb_topo}. We sketch the main changes.

\begin{proof}[Sketch of proof of Lemma~$\ref{lem:main_bis}$] Let $n_{\ell}$ be the cardinality of $V_{\ell}$. Without loss of generality, we assume that $n_1-\alt_p(\hyp_1)=\min_{\ell\in[t]}\left(n_{\ell}-\alt_p(\hyp_{\ell})\right)$. Moreover, we identify $V_{\ell}$ and $[n_{\ell}]$ in such a way that the permutation for which the minimum is attained in the definition of $\alt_p(\hyp_{\ell})$ is the identity permutation $\id$.

Let $c:E_1\times\cdots\times E_t\rightarrow[C]$ be a proper coloring of $\KG^p(\hyp_1)\times\cdots\times\KG^p(\hyp_t)$ with $C$ colors. We endow $E_1\times\cdots\times E_t$ with an arbitrary total order $\preceq$ such that $(S_1,\ldots,S_t)\preceq (T_1,\ldots,T_t)$ if $c(S_1,\ldots,S_t)<c(T_1,\ldots,T_t)$. 

Let $n=\sum_{\ell=1}^tn_{\ell}$, $\alpha=n-n_1+\alt_p(\hyp_1)+p-1$, and $m=\alpha+C-1$. We define a map $\lambda:\x\in(Z_p\cup\{0\})^{n}\mapsto(s(\x),v(\x))\in Z_p\times[m]$. To that purpose, we define for an $\x\in(Z_p\cup\{0\})^{n}$ the vectors $\y_1,\ldots,\y_t$ and the sets $A_1,\ldots,A_t$ as in the proof of Lemma~\ref{lem:main}. \\

\noindent{\bf First case: $A_{\ell}\neq Z_p$ for at least one $\ell\in[t]$.} We define 
$$\begin{array}{rcl} v(\x) & = & \ds{\sum_{\ell:\,A_{\ell}=\emptyset}\alt_{\id}(\y_{\ell}) + \sum_{\ell:\,A_{\ell}=Z_p}|\y_{\ell}|} \\ \\
&&+\ds{\sum_{\ell:\,A_{\ell}\notin\{\emptyset,Z_p\}}\left(|A_{\ell}|+\max\left\{\alt_{\id}(\widetilde{\y}_{\ell}):\,\widetilde{\y}_{\ell}\subseteq\y_{\ell}\mbox{ and }E(\hyp_{\ell}[\supp_j(\widetilde{\y_{\ell}})])=\emptyset\mbox{ for all $j$}\right\}  \right)}.\end{array}$$ 
If $A_{\ell}\in\{\emptyset,Z_p\}$ for all $\ell\in[t]$, we define $s(\x)$ to be the first nonzero entry of $\x$. Otherwise, we define $s(\x)$ to be $\varepsilon(A_1,\ldots,A_t)$. We always have in this case $1\leq v(\x)\leq\alpha$ as required.\\

\noindent {\bf Second case: $A_{\ell}=Z_p$ for all $\ell\in[t]$.} For each $j$, there is at least one $t$-tuple $(S_1(j),\ldots,S_t(j))\in E_1\times\cdots\times E_t$ and $S_{\ell}(j)\subseteq \supp_j(\y_{\ell})$ for all $j\in[p]$. Among all such $t$-tuples and over all $j$, select the one that is minimal for $\preceq$. We define then $v(\x)$ to be $\alpha+c(S_1(j),\ldots,S_t(j))$ and $s(\x)$ to be $\omega^j$ with $j\in[p]$ such that $(S_1(j),\ldots,S_t(j))$ has been chosen. Note that because of the definition of $\preceq$ and because $c(\cdot)$ is a proper coloring, $c(S_1(j),\ldots,S_t(j))\neq C$, and thus $\alpha+1\leq v(\x)\leq\alpha+C-1$, as required. \\

The map $\lambda$ satisfies the condition of Lemma~\ref{lem:zptucker}. The lower bound on $C$, and thus on the chromatic number, is a consequence of the inequality $\alpha+(m-\alpha)(p-1)\geq n$.
\end{proof}

\subsection{An application to graphs}

Two graphs $G$ and $G'$ are {\em homomorphically equivalent} when there exists a homomorphism from $G$ to $G'$ and a homomorphism from $G'$ to $G$.
Alishahi and Hajiabolhassan~\cite{AlHa14} have defined the {\em altermatic number} of a graph $G$ as the quantity 
$$\zeta(G)=\max_{\hyp}(|V(\hyp)|-\alt_2(\hyp))$$ where the maximum is taken over all hypergraphs $\hyp$ such that $\KG(\hyp)$ and $G$ are homomorphically equivalent. This definition makes sense since there always exists at least one such a hypergraph, for which actually a isomorphism holds. This hypergraph is called a {\em Kneser representation} of $G$ (see for instance~\cite{LaZi07} for a discussion on Kneser representations).

Theorem~\ref{thm:main_bis} actually shows that for any two graphs $G$ and $G'$, we also have 
$$\chi(G\times G')\geq \min\{\zeta(G), \zeta(G')\}.$$

This inequality allows to prove that Hedetniemi's conjecture is true for new families of graphs. If two graphs have their chromatic number equaling their altermatic number, then they satisfy Hedetniemi's conjecture. Usual Kneser graphs have their chromatic number equaling their altermatic number, but this equality holds for other families. By $\KG(G,H)$, we denote the Kneser graphs whose vertex set is the set of all subgraphs of $G$ isomorphic to $H$ and in which two vertices are adjacent if the corresponding subgraphs are edge-disjoint. The above mentioned equality is satisfied by the following families, which become new families satisfying Hedetniemi's conjecture:

\begin{enumerate}
\item \label{item1} the graphs $\KG(G,H)$ where $G$  is a multigraph such that the multiplicity of each edge is at least $2$ and where $H$ is a simple graph, see \cite{AlHa13bis}.
\item \label{item2} the graphs $\KG(G,rK_2)$, where $rK_2$ is a matching of size $r$, when $G$ is a sufficiently large dense graph, see \cite{AlHa14} for more details.
\item \label{item3} the graphs $\KG^2(\mathcal{B})$, where $\mathcal{B}$ is the basis set of the truncation of any partition matroid, see \cite{AlHa13} where this graph is called a {\em multiple Kneser graphs} (it is a special case of the multiple Kneser hypergraphs mentioned right after Theorem~\ref{thm:main_bis}).
\item \label{item4} the graphs $\KG(G,\mathcal{T}_n)$ (this is a new notation), when $G$ is a dense graphs: The vertices of $\KG(G,\mathcal{T}_n)$ are the spanning trees of $G$ and two vertices are adjacent if the corresponding spanning trees are edge-disjoint. For more details, see~\cite{AlHa14bis}.
\end{enumerate}

Moreover, any pair of graphs obtained by iterating the Mycielski construction on a pair of graphs with the chromatic number equaling the altermatic number satisfies Hedetniemi's conjecture. Indeed, Alishahi and Hajiabolhassan~\cite{AlHa14} proved that if $\chi(G)=\zeta(G)$, then this equality holds for the graph obtained via the Mycielski construction applied on $G$.

\begin{remark}
Partial results were already obtained by Alishahi and Hajiabolhassan in~\cite{AlHa14} for the families (\ref{item1}) and (\ref{item2}) above using the {\em strong altermatic number}. There is a general definition of it for any $r$, but we restrain it here to the case $r=2$. We introduce the following quantity for a hypergraph $\hyp=(V,E)$:
$$\salt_2(\hyp)=\min_{\pi\in\mathcal{S}_n}\max\left\{\alt_{\pi}(\x):\,\x\in(Z_2\cup\{0\})^n\mbox{ with }\,E(\hyp[\x^+])=\emptyset\mbox{ or }E(\hyp[\x^-])=\emptyset\right\},$$ where we identify $Z_2$ and $\{+,-\}$, and where $\x^+$ (resp. $\x^-$) is the set of indices $i$ such that $x_i=+$ (resp. $x_i=-$). In other words, for each permutation $\pi$ of $[n]$, we choose $\x$ such that $\alt_{\pi}(\x)$ is maximal while at most one of $\x^+$ and $\x^-$ contains an edge of $\hyp$; then, we take the permutation for which this quantity is minimal. The {\em strong altermatic number} of a graph $G$ is then the quantity 
$$\zeta_s(G)=1+\max_{\hyp}(|V(\hyp)|-\salt_2(\hyp))$$ where the maximum is taken over all hypergraphs $\hyp$ such that $\KG(\hyp)$ and $G$ are homomorphically equivalent. These authors got that for any two graphs $G$ and $G'$, we have 
$$\chi(G\times G') \geq\min\{\zeta_s(G), \zeta_s(G')\}.$$
Schrijver graphs are graphs whose chromatic numbers equal their strong altermatic numbers. This inequality provides another proof of the fact that Schrijver graphs satisfy Hedetniemi's conjecture. For some graphs of the families (\ref{item1}) and (\ref{item2}) above, the equality between the chromatic number and the strong altermatic number holds, which allows to derive yet another proof that they satisfy Hedetniemi's conjecture.
\end{remark}

\subsection*{Acknowledgment} The research of Hossein Hajiabolhassan is supported by ERC advanced grant GRACOL. Part of this work was done during a visit of Hossein Hajiabolhassan to the Universit\'e Paris Est. He would like to acknowledge Professor Fr\'ed\'eric Meunier for his support and hospitality.

\bibliographystyle{amsplain}
\bibliography{Kneser}

\providecommand{\bysame}{\leavevmode\hbox to3em{\hrulefill}\thinspace}
\providecommand{\MR}{\relax\ifhmode\unskip\space\fi MR }
\providecommand{\MRhref}[2]{%
  \href{http://www.ams.org/mathscinet-getitem?mr=#1}{#2}
}
\providecommand{\href}[2]{#2}
\begin{thebibliography}{10}

\bibitem{AlHa13bis}
M.~Alishahi and H.~Hajiabolhassan, \emph{Chromatic number via tur\'an number},
  preprint.

\bibitem{AlHa14}
\bysame, \emph{Hedetniemi's conjecture via alternating chromatic number},
  preprint.

\bibitem{AlHa14bis}
\bysame, \emph{On chromatic number and minimum cut}, preprint.

\bibitem{AlHa13}
\bysame, \emph{On chromatic number of {K}neser hypergraphs}, preprint.

\bibitem{AlDrLu09}
N.~Alon, L.~Drewnowski, and T.~{\L}uczak, \emph{Stable {K}neser hypergraphs and
  ideals in $\mathbb{N}$ with the {N}ikod\'ym property}, Proceedings of the
  {A}merican Mathematical Society \textbf{137} (2009), 467--471.

\bibitem{AlFrLo86}
N.~Alon, P.~Frankl, and L.~Lov\'asz, \emph{The chromatic number of {K}neser
  hypergraphs}, Transactions of the American Mathemathical Society \textbf{298}
  (1986), 359--370.

\bibitem{Do09}
A.~Dochtermann, \emph{Hom complexes and homotopy theory in the category of
  graphs}, European Journal of Combinatorics \textbf{30} (2009), 490--509.

\bibitem{Do88}
V.~L. Dol'nikov, \emph{A certain combinatorial inequality}, Siberian
  Mathematics Journal \textbf{29} (1988), 375--397.

\bibitem{DoWa80}
W.~D\"orfler and D.~A. Waller, \emph{A category-theoretical approach to
  hypergraphs}, Archiv der Mathematik \textbf{34} (1980), 185--192.

\bibitem{Er76}
P.~Erd\H{o}s, \emph{Problems and results in combinatorial analysis}, Colloquio
  Internazionale sulle Teorie Combinatorie (Rome 1973), Vol. II, No. 17 in Atti
  dei Convegni Lincei, 1976, pp.~3--17.

\bibitem{Ha10}
H.~Hajiabolhassan, \emph{On the $b$-chromatic number of {K}neser graphs},
  Discrete Applied Mathematics \textbf{158} (2010), 232--234.

\bibitem{He79}
P.~Hell, \emph{An introduction to the category of graphs}, Topics in graph
  theory (New York~Acad. Sci, ed.), 1979.

\bibitem{Kr92}
I.~K\v{r}\'i\v{z}, \emph{Equivariant cohomology and lower bounds for chromatic
  numbers}, Transactions Amer. Math. Soc. \textbf{33} (1992), 567--577.

\bibitem{KR00}
\bysame, \emph{A correction to ``{E}quivariant cohomology and lower bounds for
  chromatic numbers''}, Transactions of the American Mathematical Society
  \textbf{352} (2000), 1951--1952.

\bibitem{LaZi07}
C.~Lange and G.~M. Ziegler, \emph{On generalized {K}neser hypergraph coloring},
  Journal of Combinatorial Theory, Series A \textbf{2007} (2007), 159--166.

\bibitem{Lo79}
L.~Lov\'asz, \emph{Kneser's conjecture, chromatic number and homotopy}, Journal
  of Combinatorial Theory, Series A \textbf{25} (1978), 319--324.

\bibitem{Mat04}
J.~Matou\v{s}ek, \emph{A combinatorial proof of {K}neser's conjecture},
  Combinatorica \textbf{24} (2004), 163--170.

\bibitem{Me11}
F.~Meunier, \emph{The chromatic number of almost-stable {K}neser hypergraphs},
  Journal of Combinatorial Theory, Series A \textbf{118} (2011), 1820--1828.

\bibitem{Me14}
\bysame, \emph{Colorful subhypergraphs in {K}neser hypergraphs}, Electronic
  Journal of Combinatorics \textbf{21} (2011).

\bibitem{SiTa06}
G.~Simonyi and G.~Tardos, \emph{Local chromatic number, {K}y {F}an's theorem,
  and circular colorings}, Combinatorica \textbf{26} (2006), 587--626.

\bibitem{SiTa07}
\bysame, \emph{Colorful subgraphs of {K}neser-like graphs}, European Journal of
  Combinatorics \textbf{28} (2007), 2188--2200.

\bibitem{St76}
S.~Stahl, \emph{$n$-tuple colorings and associated graphs}, Journal of
  Combinatorial Theory, Series B \textbf{20} (1976), 185--203.

\bibitem{Ta08}
C.~Tardif, \emph{Hedetniemi's conjecture, 40 years later}, Graph Theory Notes
  N. Y. \textbf{54} (2008).

\bibitem{VaVe05}
M.~Valencia-Pabon and J.-C. Vera, \emph{On the diameter of {K}neser graphs},
  Discrete Mathematics \textbf{305} (2005), 383--385.

\bibitem{VaVe06}
\bysame, \emph{Independence and coloring properties of direct products of some
  vertex-transitive graphs}, Discrete Mathematics \textbf{306} (2006),
  2275--2281.

\bibitem{Zh92}
X.~Zhu, \emph{On the chromatic number of the products of hypergraphs}, Ars
  Combinatoria \textbf{34} (1992), 25--31.

\bibitem{Zh98}
\bysame, \emph{A survey on {H}edetniemi's conjecture}, Taiwanese Journal of
  Mathematics \textbf{2} (1998), 1--24.

\bibitem{Zi02}
G.~M. Ziegler, \emph{Generalized {K}neser coloring theorems with combinatorial
  proofs}, Inventiones Mathematicae \textbf{147} (2002), 671--691.

\end{thebibliography}

\end{document}